\documentclass{article}
  \usepackage[english]{babel}
\usepackage{amsthm,amsfonts, amsbsy, amssymb,amsmath,graphicx}
\usepackage{graphics}

  \def\thtext#1{
  \catcode`@=11
  \gdef\@thmcountersep{. #1}
  \catcode`@=12}

\newtheorem{thm}{Theorem}
\newtheorem{lm}{Lemma}
\newtheorem{ex}{Example}
\newtheorem{crl}{Corollary}
\newtheorem{st}{Statement}

 \newcounter{il}[section]
 \renewcommand{\thtext}
 {\thechapter.\arabic{il}}

 \newcounter{il2}[section]
 \renewcommand{\thtext}
 {\thechapter.\arabic{il2}}

 \newenvironment{dfn}{\trivlist \item[\hskip\labelsep{\bf Definition}]
 \refstepcounter{il2}{\bf \arabic{il2}.}}%
 {\endtrivlist}

 \newenvironment{rk}{\trivlist \item[\hskip\labelsep{\bf Remark}]
 \refstepcounter{il2}{\bf \arabic{il2}.}}%
 {\endtrivlist}

\def\Z{{\mathbb Z}}
\def\N{{\mathbb N}}
\def\R{{\mathbb R}}

\date{}

\author{Vassily Olegovich Manturov}

\title{Parity and Cobordisms of Free Knots}

\begin{document}

\newcommand{\eps}{\varepsilon}

\maketitle

\begin{abstract}
In the present paper, we construct a simple invariant which provides
a sliceness obstruction for {\em free knots}. This obstruction
provides a new point of view to the problem of studying cobordisms
of curves immersed in $2$-surfaces, a problem previously studied by
Carter, Turaev, Orr, and others.

The obstruction to sliceness is constructed by using the notion of
{\em parity} recently introduced by the author into the study of
virtual knots and their modifications. This invariant turns out to
be an obstruction for cobordisms of higher genera with some
additional constraints.


\end{abstract}

\section{Introduction}

Curves immersed in 2-surfaces admit a natural notion of {\em
cobordisms}: one says that two curves $\gamma$ and $\gamma'$
immersed in  oriented closed $2$-surfaces (not necessarily
connected, but for the sake of simplicity we shall use the notation
for connected surfaces) $S_{g}$ and $S_{g'}$, are cobordant if there
is an oriented $3$- manifold $M$ whose boundary consists of two
surfaces $S^{+}_{g}=S_{g}$ and $S^{-}_{g'}=S_{g'}$ and a properly
{\em smoothly mapped cylinder} $C\subset M$, such that $\partial
C=C\cap
\partial M$ and $C\cap S^{+}_{g}=\gamma,C\cap S^{-}_{g'}=\gamma'$, see Fig. \ref{cobor}.
\begin{figure}
\centering\includegraphics[width=200pt]{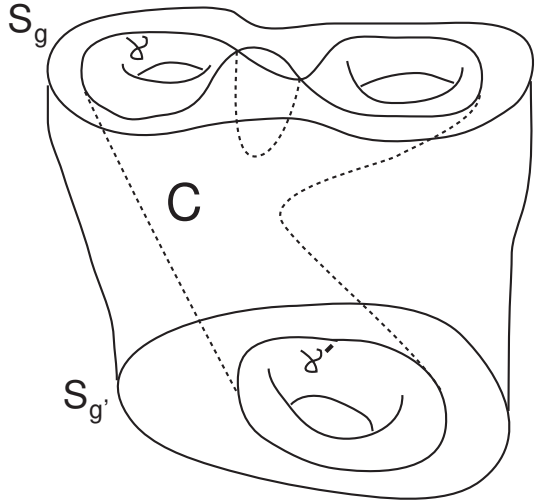} \caption{A
cobordism between immersed curves} \label{cobor}
\end{figure}
 In
particular, one says that a curve $\gamma\in S_{g}$ is {\em
null-cobordant} (or {\em slice}) whenever there is an oriented
$3$-manifold $M,\partial M=S_{g}$ and a disc $D$ properly immersed
in $M$ by a map $f$ such that $f(\partial D)=f(D)\cap \partial
 M=\gamma$.

Analogously, one says that the {\em slice genus} of $\gamma\subset
S_{g}$ does not exceed $h$ if in the above definition one uses a
surface $D_{h}$ of genus $h$ with one boundary component instead of
the disc $D$.

The cobordism is obviously an equivalence relation; cobordism
classes of curves form a group (with connected sum playing the role
of multiplication) where the equivalence class of null-cobordant
elements play the role of unity.

The first obstructions for curves to be null-cobordant were found by
Carter, \cite{Ca}; after that, the theory was also studied by Turaev
\cite{Tu}, Orr, and others.

\begin{rk}
In the sequel, we deal only with generic immersions of curves in
$2$-surfaces, unless specified otherwise.
\end{rk}

It can be easily proved that if two curves $\gamma,\gamma'\subset
S_{g}$ are homotopic then they are cobordant, so one can talk about
{\em cobordism classes of homotopy classes of curves}.

Thus, it is natural to talk about cobordism classes of {\em flat
virtual knots} (virtual strings) \cite{Flats}, which are equivalence
classes of circle immersions in $2$-surfaces considered up to
homotopy and stabilization/destabilization. The stabilization
operation (addition of a handle to $S_g$  away from the curve)
obviously does not change the cobordism class of the curve.

The paper  \cite{Ma1} (full versions see in \cite{Ma}) pioneered the
overall study of {\em free knots}, which are formal equivalence
classes of framed $4$-graphs with modulo the three Reidemester
moves. The free knots are a thorough simplifiaction of virtual
knots, moreover, the are a simplification of flat virtual knots as
well, since they can be obtained from the latter by ``forgetting the
surface structure''.

In  \cite{Ma1}, the first examples of non-trivial free knots were
constructed, and some properties were established. Some days
afterwards, non-trivial free knots were constructed by Gibson
\cite{Gib}. The key notion used in many constructions from
\cite{Ma,Ma1,Atoms} is the notion of {\em parity}, which allows one
to solve several problems, construct new invariants of (virtual)
knots and their relatives, strengthen some known invariants, and
construct maps from knots to knots.

Once it is proved that free knots are generally non-trivial, the
question of {\em non-sliceness} of free knots arose; rigorous
definitions of the slice genus see ahead in Definition
\ref{frknslice}.

Informally speaking, before constructing an invariant of some
topological objects, one should first look at some ``building
bricks'' or ``nodes'' of these objects (like crossings of a knot
diagram or intersection points of a curve generically immersed in a
$2$-surface or intersection lines of a $2$-surfaces immersed in a
$3$-manifold), and try to figure out whether one can distinguish
between different types of them by looking at the global
topology/combinatorics of our object. Then, one may try to modify
existing combinatorial invariants (or to construct new ones) by
adding this extra information into a given setup.

In the context of free knots all vertices of the four-graph can be
naturally split into the set of ``even ones'' and ``odd ones''.

It is clear that every invariant of free knots generates an
invariant of flat virtual knots (one just has to take its
composition with a natural ``forgetting projection''). So, for free
knots one can naturally define the notion of cobordism in a way such
that that if two flat virtual knots are cobordant then so are the
underlying free knots.

On the other hand, cobordism invariants constructed by Carter, Orr,
Turaev, and others can not be straightforwardly defined for the case
of free knots, since they use some homological data of the surface,
which free knots do not possess. In some sense, parity can
substitute homological/homotopy information, when there is no
``genuine'' homology.

The concept of parity has some other applications in the cobordism
theory for free knots/immersed curves. In particular, if a free knot
$\gamma$ is slice, then so is the free knot $\gamma'$ obtained from
$\gamma$ by ``killing odd crossings'', see Theorem \ref{deleteodd}
ahead.

The aim of the present paper is to construct one simple (in fact,
integer-valued) invariant of free knots which gives an obstruction
for a free knot to be slice (null-cobordant).

The paper is organized as follows. First we define free knots,
parity, Gaussian parity and construct our invariant. We prove its
invariance under Reidemeister moves. Then, to show that the
invariant is well-behaved under cobordisms, we have to extend the
notion of parity from one knot to the spanning disc of the
cobordism. This is done by marking double lines of the spanning disc
as ``even'' and ``odd''. After that we give the basic definitions of
Morse theory for cobordisms, and outline the proof of the main
theorem. Taking a Morse function on a spanning disc, one extends the
invariant to all the regular sections of this function, and the
non-triviality of the initial invariant coupled with simple Morse
theoretic arguments leads to a contradiction. The key point in the
proof is the way to extend the notion of parity from vertices of
$4$-graphs to double lines of $2$-surfaces.

We conclude the paper by a list of further possible developments of
this theory of parity and cobordisms. In particular, the approach
described in the present paper is good for detecting non-sliceness
but is not (in its original form) applicable for getting estimates
of the slice genus because of some caveats relating parity and Morse
theory.

On the other hand, the concept of parity evidently has
multidimensional analogues and may be applied for $2$-knots in
$4$-space and similar objects. Our last section initiates a
discussion of  problems of such sort.

In \cite{IM} some other relation called {\em cobordisms of free
knots} was discussed: instead of topological definition obtained by
considering {\em spanning discs}, we dealt with a formal {\em
combinatorial} definition following Turaev which relies on a set of
{\em moves}. An invariant of combinatorial cobordism was
constructed.

The interrelation between these two equivalence relations,
topological cobordisms and combinatorial cobordisms, will be
considered in a future publication.

\subsection{Acknowledgement}

My study of free knots and their cobordisms was initiated after I
got deeply impressed by a talk ``Surface Knots and Iterated
Intersection Pairings'' by Kent E.Orr given at a conference in
Heidelberg in 2008 about a deep connection between cobordisms of
curves in $2$-surfaces and cobordisms of knots. The obstruction (due
to Carter and later, its expansion due to Turaev and Orr) was based
on homological information about the surface in question. This lead
me to an idea of finding a good ``substitute'' for homology when
instead of a curve in a $2$-surface one has an abstract curve with
self-intersection.

Various discussions with Kent Orr led me to a deeper understanding
of many results in this area, and I am deeply indebted to him.

I express my gratitude to L.H.Kauffman, D.P.Ilyutko and M.Chrisman
for fruitful consultations.

\subsection{Free Knots}

By a {\em graph} we always mean a finite (multi)graph; loops and
multiple edges are allowed.
 From now on, by a
{\em $4$-graph} we mean the following generalization of a
four-valent graph: a $1$-dimensional complex $\Gamma$, with each
connected component being homeomorphic either to the circle (with no
matter how many $0$-cells) or to a four-valent graph; by a {\em
vertex} we shall mean only vertices of those components which are
homeomorphic to four-valent graphs, and by {\em edges} we mean
either edges of four-valent-graph-components or circular components;
the latter will be called {\em cyclic edges}.

We say that a $4$-graph is {\em framed} if for every vertex of it,
the four emanating half-edges are split into two pairs. We call the
half-edges from the same pair {\em opposite}. We shall also apply
the term {\em opposite} to edges containing opposite half-edges.

By  {\em isomorphism} of framed $4$-graphs we assume a
framing-perserving homeomorphism. All framed  $4$-graphs are
considered up to isomorphism.

Denote by $G_{0}$ the framed $4$-graph homeomorphic to the circle.

By a {\em unicursal component} of a framed $4$-graph we mean either
its connected component homeomorphic to the circle or an equivalence
class of its edges, where the equivalence is generated by the
relation of being opposite.

By a {\em chord diagram} we mean a cubic graph consisting of one
selected cycle  passing through all vertices of the graph and a set
of chords. We call this cycle {\em the core of the chord diagram}. A
chord diagram is {\em oriented} whenever its circle is oriented.
Edges belonging to the core cycle are called {\em arcs} of the chord
diagram.

Let $D$ be a chord diagram. Then the corresponding $4$-graph $G(D)$
with a unique unicrusal component is constructed as follows. If the
set of chords of $D$ is empty then the corresponding graph will be
$G_{0}$. Otherwise, the edges of the graph are in one-to-one
correspondence with arcs of the chord diagram, and vertices are in
one-to-one correspondence with chords of the chord diagram. The arcs
incident to the same chord end, correspond to the (half)-edges which
are formally opposite at the vertex corresponding to the chord. We
say that two chords $a$ and $b$ of a chord diagram $D$ are {\em
linked}, if the ends of the chord  $b$ belong to different connected
components of the complement to the ends of $a$ in the core circle
$C$. In this case we write $\langle a,b\rangle=1$. Otherwise we say
that chords are {\em unlinked} and write $\langle a,b\rangle=0$.

The inverse procedure (of constructing a chord diagram from a framed
$4$-graph) with one unicursal component is evident. In this
situation every connected framed $4$-graph can be considered as a
topological space obtained from the circle by identifying some pairs
of points. Thinking of the circle as the core circle of a chord
diagram, where the pairs of identified chords will correspond to
chords, one obtains a chord diagram.

For a framed $4$-graph  $G$ with one unicursal component we define
the $\Z_{2}$-pairing of vertices
 vertices $v_{1},v_{2}$ as follows: $\langle
v_{1},v_{2}\rangle=\langle a(v_{1}),a(v_{2})\rangle$, where
$a(v_{1}),a(v_{2})$ are chords corresponding to the vertices $v_{1},
v_{2}$.

Our first aim is to study equivalence classes of framed four graphs
modulo some moves corresponding to Reidemeister moves for knots, cf.
\cite{Ma}.

Each of this moves is a transformation of one fragment of a framed
$4$-graph.

{\em Graphical notation.} In figures depicting moves on diagrams, we
shall draw only the changing part; the stable part will be omitted.
In the case of one unicursal component a move can be represented by
Gauss diagram; the move changes the diagram on some set of arcs; we
shall not draw those chords away from the Reidemeister move being
performed; the arcs having no ends of chords taking part in the
move, will be depicted by dotted lines.

When drawing framed graphs on the plane, we always assume that the
opposite half-edges structure is induced from the plane.

The first Reidemeister move is an addition/removal of a loop, see
Fig.\ref{1r}.

\begin{figure}
\centering\includegraphics[width=100pt]{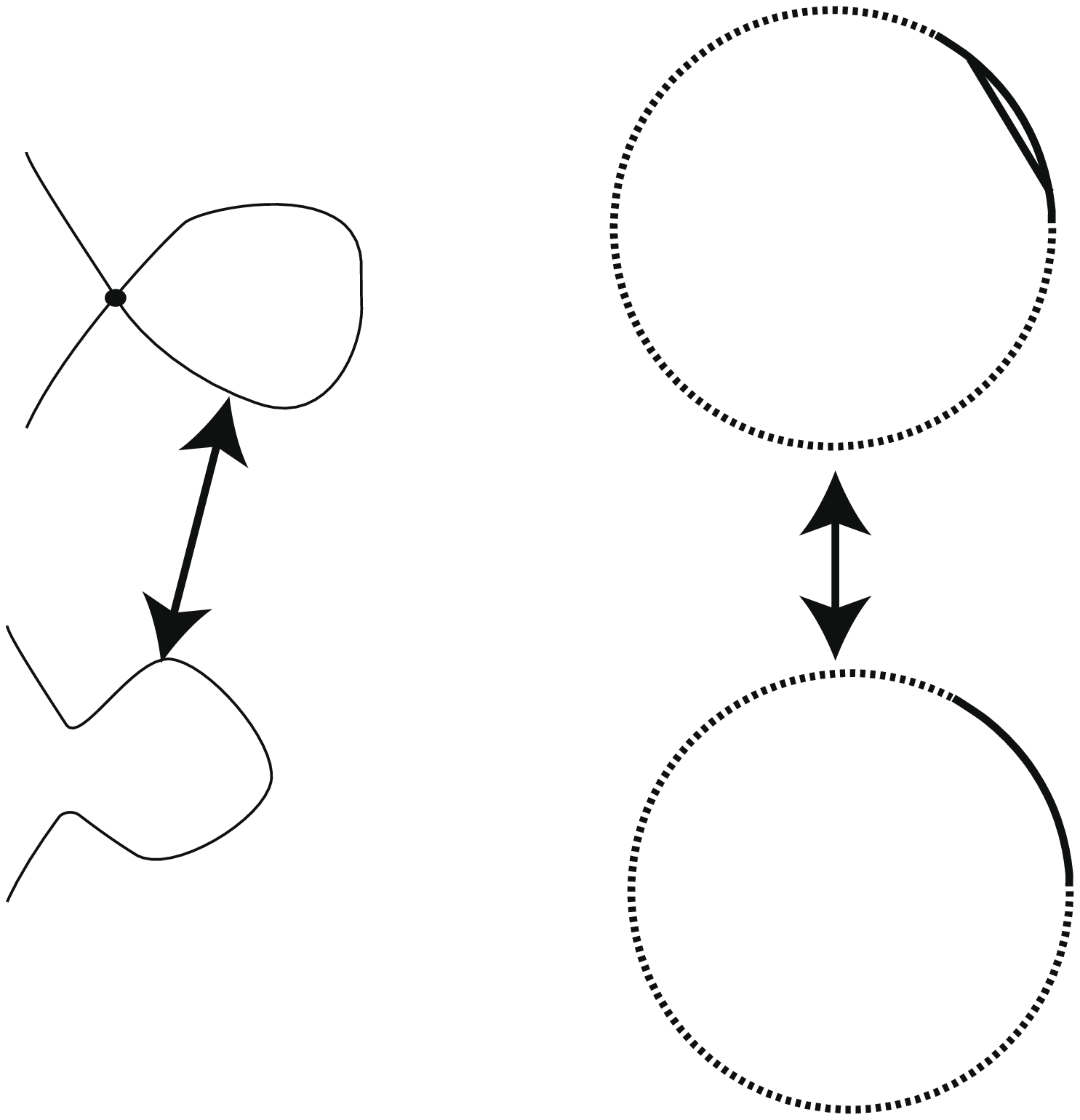} \caption{The first
Reidemeister move and its chord diagram version} \label{1r}
\end{figure}

The second Reidemeister move is an addition/removal of a bigon
formed by a pair of edges which are adjacent (not opposite) in each
of the two vertices, see Fig. \ref{2r}. It has two variants;
nevertheless, it can be easily seen, that any of these two variants
is expressed as a combination of the other one and the second
Reidemeister moves.

\begin{figure}
\centering\includegraphics[width=150pt]{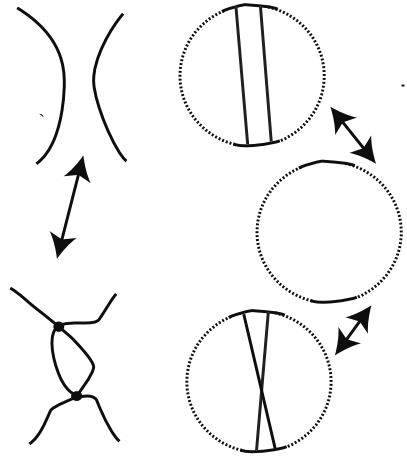} \caption{The second
Reidemeister move and its chord diagram version} \label{2r}
\end{figure}

The third Reidemeister move is shown in Fig. \ref{3r}.

\begin{figure}
\centering\includegraphics[width=150pt]{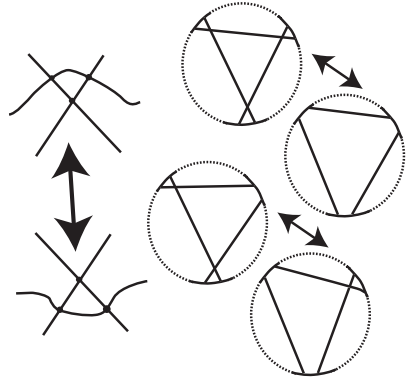} \caption{The third
Reidemeister move and its chord diagram version} \label{3r}
\end{figure}

\begin{dfn} A  {\em free link} is an equivalence class of framed 4-graphs modulo Reidemeister moves.

It is evident that the number of components of a framed 4-graph does
not change after applying a Reidemeister move, so, it makes sense to
talk about {\em the number of components of a free link}.

By a {\em free knot} we mean a free link with one unicursal
component.

Free knots can be treated as equivalence classes of Gauss diagrams
by moves corresponding to Reidemeister moves.
\end{dfn}

The {\em free unknot} (resp., free $n$-component unlink) is the free
knot (link) represented by $G_{0}$ (resp., by $n$ disjoint copies of
$G_{0}$).

Analogously one defines {\em long free knots}; here one should break
the only component and pull its ends to infinity; all graphs; all
moves are considered in finite domains. Equivalently, we may
consider a $4$-valent graph with one unicursal component with a
marked edge up to the Reidemeister moves performed away from the
marked point.

\subsection{The Parity Axiomatics. The Gaussian Parity}

We define the key properties of {\em parity} for free knots by using
a set of axioms, following \cite{Ma}.

\begin{dfn}
A {\em parity} is a map assigning to each pair $(\Gamma,V)$, where
$\Gamma$ is a framed $4$-graph and $V$ is a vertex of $\Gamma$, a
number $p(V)$ (also denoted by $p_{\Gamma}(V)$) which is equal to
$0\in \Z_{2}$ (in this case the vertex $V$ is called {\em even}) or
$1\in \Z_{2}$ (when the vertex is called {\em odd}) in such a way
that this rule satisfies the following axioms.

\begin{enumerate}

\item If a framed $4$-graph $K_2$ is obtained from a framed $4$-graph $K_1$ by a first (decreasing) Reidemeister move then
the crossing of $K_1$ taking part in the Reidemeister move is  {\em
even};

\item If $K_2$ is obtained from $K_1$ by a second Reidemeister move
then both crossings participating in this Reidemeister move, are
{\em of the same parity};

\item If $K_2$ is obtained from $K_1$ by a third Reidemeister move
then there is a one-to-one correspondence between the triple of
crossings of  $K_1$ taking part in the Reidemeister move and the
analogous tripe of crossings of $K_2$, ($(x,x'),(y,y'),(z,z')$), see
Fig. \ref{sootvt}.

\begin{figure}
\centering\includegraphics[width=200pt]{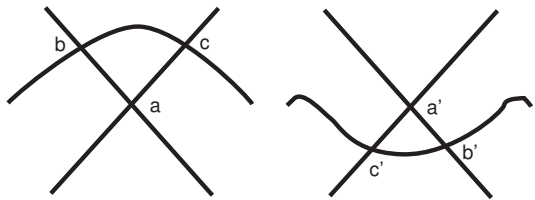} \caption{The
third Reidemeister move and corresponding crossings} \label{sootvt}
\end{figure}

We require that

a) $p(x)=p(x'),p(y)=p(y'), p(z)=p(z')$;

b) Among $x,y,z$, the number of odd crossings is even (i.e., is
equal to $0$ or  $2$).

\item For every Reidemeister move $K_{1}\to K_{2}$ there exists a
one-to-one correspondence between crossings of $K_1$ away from the
Reidemeister move, and crossings of $K_2$ away from the Reidemeister
move.

We require that the corresponding crossings of the diagrams $K_1$
and $K_2$ are of the same parity.

\end{enumerate}

\end{dfn}

We introduce the notion of {\em justified parity} analogously to the
notion of parity.

\begin{dfn}
By {\em justified parity} of crossings we mean a parity where each
odd crossing is marked by a letter $b$ or $b'$ (in these cases we
call a crossing {\em an odd crossing of the first type} or {\em an
odd crossing of the second type}, respectively), so that the
following holds:

\begin{enumerate}

\item If a second Reidemeister move is applied to two odd crossings then they are of the same type (either both are $b$ or both are $b'$).

\item If in a third Reidemeister move we have two odd crossings, then each of them changes its type after the Reidemeister move
is applied (the crossing marked by $b$ before the Reidemeister move
is applied should correspond to the crossing marked by $b'$ after
the Reidemeister move is applied).

\item Moreover, odd crossings not taking part in the Reidemeister move, do not change their types.

\end{enumerate}

\end{dfn}

We define the {\em Gaussian parity} and the {\em Gaussian justified
parity} on framed $4$-graphs with one unicursal component, as
follows.

\begin{dfn}[The Gaussian parity and justified parity]
Let $D$ be a chord diagram.

We say that a chord of $D$ is {\em even} (after Gauss), if the
number of chords linked with it, is even, and {\em odd}, otherwise.
Furthermore, an odd chord is {\em of the first type} if it is linked
with an even number of even chords; otherwise an odd chord is said
to be {\em of the second type} (after Gauss).

For a framed  $4$-graph $\Gamma$ corresponding to a chord diagram
 $D$ the Gaussian parity and justified parity are defined as those
 of the corresponding chord diagram.

\end{dfn}

It can be easily checked that the Gaussian parity and the Gaussian
justified parity satisfy the axioms of parity and justified parity.

Later in this paper, for cobordism purposes we shall then extend the
notion of Gaussian parity and justified Gaussian parity for another
situation. To that end, we shall first define the Gaussian parity
for double lines of a $2$-disc with generic intersections, and then
for every regular section of this disc (which will be a framed
$4$-graph representing a free link) we shall define the parity for
crossings to be the parity of double lines it comes from.

However, for the first goal (the construction of an invariant of
free knots) it would be sufficient for us to have a well defined
parity just for framed $4$-graphs with one unicursal component.

\subsection{A functorial mapping $f$}

Let $K$ be a framed $4$-graph. Let $f$ be a diagram obtained from
$K$ by {\em removing all odd crossings} and connecting opposite
edges of former crossings to one edge.

\begin{thm}
The map $f$ is a well-defined map on the set of all free knots. For
a virtual knot diagram $K$, $f(K)=K$ iff all crossings of $K$ are
even. Otherwise, the number of classical crossings of $f(K)$ is
strictly less than the number of classical crossings of
$K$.\label{vyro}
\end{thm}

This theorem was known to Turaev for the case of Gaussian parity; in
the general case it easily follows from the parity axioms.

As we shall see, this map will take slice free knots to slice free
knots and will never increase the slice genus.

\section{An invariant of free knots}

In the present section, we shall construct an invariant of free
knots constructed from the parity and justified parity, and prove
its invariance. We shall later prove that this invariant delivers a
sliceness obstruction for a free knot. Within the present section,
by {\em parity} and {\em justified parity} we mean the Gaussian
parity and the Gaussian justified parity.

An extension of the invariant to be presented below, is constructed
in \cite{MM}; for our purposes (sliceness obstruction) the version
given here will suffice, nevertheless, it is an important question
to investigate the invariant from \cite{MM} related to the
sliceness.

Set $G=\langle a,b,b'|a^{2}=b^{2}=b'^{2}=e,ab=b'a\rangle$.

Our first goal is to construct an invariant of free long knots
(resp., of compact free knots) valued in $G$ (in the set of
conjugacy classes of the group $G$).

Let $D$ be an oriented chord diagram, with a marked point $X$ on the
core circle $C$ distinct from any chord end. Later we shall see how
one can get rid of the orientation of $D$.

We distinguish between {\em even} and {\em odd} chords of $D$;
moreover, we distinguish between {\em two types of odd chords} of
$D$.

With a marked oriented chord diagram $(D,X)$ we associate a word in
the alphabet $\{a,b,b'\}$, as follows. Let us walk along the core
circle $C$ of the diagram $D$ starting from $X$. Every time we meet
a chord end, we write down a letter $a,b$ or $b'$ depending on
whether the chord whose endpoint we met, is even, first type odd, or
second type odd. Having returned to the point $X$, we obtain a word
$\gamma(D,X)$; this word determines an element of $G$; by abuse of
notation we shall denote this element by $\gamma(D,X)$ as well.
Moreover, sometimes we shall omit $X$ from the notation when it is
clear from the context which initial point we have chosen.

\begin{thm}
If two marked chord diagrams $(D,X)$ and $(D',X')$ generate
equivalent free knots then the following two conjugacy classes
coincide: $[\gamma(D)]=[\gamma(D')]$ in $G$.\label{th1}
\end{thm}

An extended version of this theorem is proved in \cite{MM}; see also
\cite{}

\begin{proof}

Indeed, if  $D$ and $D'$ differ by one first Reidemeister move (say,
$D'$ has one extra chord) then the word $\gamma(D')$ is obtained
from $\gamma(D)$ by an addition of two consequent letters $a\cdot
a$; thus, the corresponding elements from $G$ coincide.

Analogously, if  $D'$ obtained from $D$ by an increasing second
Reidemeister move, then the two new chords of  $D'$ are of the same
parity and of the same type; denote the letter corresponding to each
of these two chords ($a$,$b$ or $b'$), by $u$. Thus, the word
$\gamma(D',X')$ is obtained from $\gamma(D,X)$ by addition of
 $u\cdot u$ in two places. As in the first case, it does not change
 the corresponding element of $G$.

When applying the third Reidemeister move $D\to D'$ one of the
following two possibilities may occur. In the first case,
 all three chords participating in the third Reidemeister move, are even.

In this case the words $\gamma(D)$ and $\gamma(D')$ coincide
identically.

In the second case, two of the three chords taking part in the
Reidemeister move, are odd, and one chord is even. Consider those
three segments of the words $\gamma(D)$ and $\gamma(D')$ where the
ends of the three moving chords are located. For those segments
containing an end of the odd chord, we get one of the two
substitutions $ab\longleftrightarrow b'a$ or  $ba\longleftrightarrow
ab'$: indeed, after the third Reidemiester move, the odd chord has
changed its type.

Both changes correspond to some relations in  $G$.

Now consider the segment of the diagram containing the two ends of
the odd chords. If these two odd chords are of the same type in $D$,
then in $D'$ they are of the same type as well. Consequenlty, when
passing from $\gamma(D)$ to $\gamma(D')$ we replace the subword
$b\cdot b$ by $b'\cdot b'$ or vice versa. Since both subwords
correspond to the trivial element of  $G$, we have
$\gamma(D)=\gamma(D')$ in $G$.

Finally, if the two chords participating in the third Reidemeister
move, are of different types on the diagram $D$, then when passing
from $\gamma(D)$ to $\gamma(D')$ the order of $b$, $b'$ in the
fragment of the corresponding word {\em stays the same}: the
adjacent letters
 $b$ and $b'$ change their position twice.

Thus, no Reidemeister move changes the element of $G$ corresponding
to the chord diagram.
\end{proof}

This theorem immediately yields the following

\begin{crl}
The conjugacy class of the element $[\gamma(D,X)]$ in $G$ is an
invariant of free knots given by the diagram $D$, i.e., it does not
depend on the fixed point $X$.\label{crl1}
\end{crl}

Indeed, moving the marked point through a chord end corresponds to a
cyclic permutation of the letters, which, in turn, generates a
conjugation in $G$.

\subsection{The Cayley graph of $G$}

Its Cayley graph looks like a vertical strip on a squared paper
between $x=0$ and $x=1$: we choose the point $(0,0)$ to be the unit
in the group; the multiplication by  $a$ on the right is chosen to
be one step in a horizontal direction (to the right if the first
coordinate of the point is equal to zero, and to the left if this
first coordinate is equal to one), the multiplication by $b$ is one
step upwards if the sum of coordinates is even and one step
downwards if this sum is odd, and the multiplication by  $b'$ is one
step downwards if the sum of coordinates is even and one step
upwards if the sum of coordinates is even, see Fig. \ref{band}.

\begin{figure}
\centering\includegraphics[width=40pt]{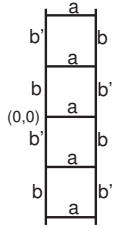} \caption{The Cayley
graph of the group $G$} \label{band}
\end{figure}

With each pointed chord diagram one associates an element from $G$
having coordinates $(0,l)$. Moreover, the conjugacy class of the
element $(0,l)$ for $l\neq 0$ consists of the two elements: $(0,4l)$
and $(0,-l)$. Thus, for each {\em long free knot} one gets an
integer-valued invariant, equal to  $l$; we shall denote this
invariant by  $l(K)$; each compact free knot has, in turn, the
invariant equal to  $|l|$; we shall denote the latter by $L(K)$.

It is an easy exercise to show that $l$ is divisible by $4$; in
fact, it can be also shown that $l$ is divisible by $8$.

It is obvious that if we invert the orientation of the chord
diagram, we shall reverse the order of letters in the word $\gamma$;
this leads to the switch $(0,l)\to (0,-l)$. So, the invariant $L(K)$
can be defined for {\em unoriented} free knots.

\begin{figure}
\centering\includegraphics[width=150pt]{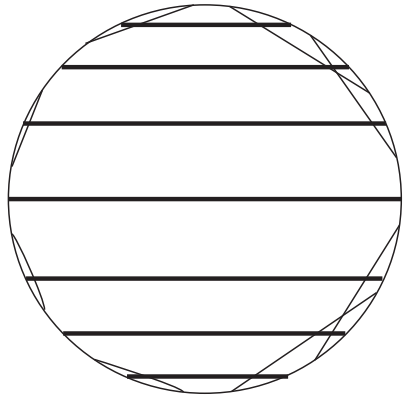} \caption{A non-slice
free knot} \label{fig1}
\end{figure}

We shall describe elements of $G$ by two coordinates on the Cayley
graph. It can be easily checked that any element of $G$
corresponding to a chord diagram has coordinates $(0,4m)$. Indeed,
the word corresponding to a Gauss diagram has an even number of
occurences of $a$, and the total number of letters $b$ and $b'$ is
divisible by four. The latter, in turn,  follows from the fact that
the number of {\em odd chords} of a chord diagram is even, which, in
turn, just means that the number of odd-valent vertices of any graph
is always even. Moreover, the conjugacy class of the element with
coordinates $(0,4m)$, $m\neq 0$ in $G$ consists of two elements:
$(0,4m)$ and $(0,-4m)$. Thus, every {\em compact free knot} has an
integer-valued invariant equal to  $L=|4m|$  (it will be convenient
for us to preserve the factor $4$ in the definition of our
invariant).

In Fig.  \ref{fig1} we have a free knot $K_{1}$ for which
$L(K_{1})=16$.

We use {\em bold} lines to describe even chords. The corresponding
word in $G$ (with an appropriate choice of the marked point) looks
like $(b'a)^{7}b'b(ab)^{7}= (b'b)^{16}$.

\subsection{Remarks on the definition of the invariant $L$ for links}

Note that Theorem \ref{th1} works for {\em any parity}, not only for
the Gaussian one.

For cobordism purposes, we have to understand the behaviour of the
invariant $L$ not only under Reidemeister moves, but also under
Morse bifurcations.

Our further strategy is as follows. Assuming we have a cobordism
(see Definition \ref{frknslice}) ${\cal D}\to D$ that spans the
curve $\gamma$, we shall define the parity and justified parity of
{\em double lines} (i.e., those lines on $D$ having preimages
consisting of two connected components). This parity will be defined
in a way such that the parity (the justified parity) for a double
point on $\gamma$ will coincide with the parity (justified parity)
for a double line, this point belongs to. Besides, this approach
will allow us to define the parity and the justified parity for any
{\em generic section} of a cobordism; such a section will be a
framed $4$-graph representing a multicomponent link. With these
parity and justified parity in hand, we shall be able to extend our
invariant $L$ to sections of $L$ (levels of the Morse function on
$D$) and then understand the behaviour of this invariant under Morse
bifurcations.

For a single component free-link (free knot), the value of the
invariant $\gamma$ can be expressed by one non-negative integer $L$.
For a cobordism, every section is a multicomponent free link, so, we
have to define the invariant $\gamma$ on this multicomponent link to
be a collection of conjugacy classes of elements of $G$ (one for
each component), and we may require that these elements of $G$ are
expreseed by one non-negative integer number each, in order to
associate an integer number to each component.

To this end, we shall need that:

1) The parity and the relative parity are well-defined for sections
and well-behaved under Morse moves.

2) The number of even intersection points on each component is {\em
even}; otherwise the corresponding element of $G$ will have a
non-zero first coordinate.

3) The value of the first coordinate of the element of $G$
(described by the invariant $L$) behaves well with respect to any
Morse bifurcations.

In particular, every component of a non-singular level link has an
even number of intersection points: it is necessary to describe the
value of our invariant $L$ for the parity to be well defined.
Indeed, in order to define the Gaussian parity of some crossing, one
has to take some ``half'' of the circle corresponding to this
crossing and count the number of intersection points belonging to
this half. In the case of a free knot the parity of this number of
points does not depend on the half one chooses because it equals the
parity of the number of chords linked with the chord in question.

When we have a two-component link, and we take a crossings formed by
single component,  the two parities corresponding to the two halves
will be different if the total number of crossings between
components is odd.

So, for those $2$-component links having an odd number of
intersection points between components, there is no immediate way to
extend the Gaussian parity.

As we shall see further, all these conditions will be automatically
satisfied when we consider the {\em sections} of a disc cobordism.

\section{Parity as homology}

In the present section, we are going to reformulate the notion of
parity and justified parity in terms of homology,
\cite{Paritytrieste}. This reformulation will be useful for
understanding the way how to define the parity on double lines of
the cobordism.

Consider a framed $4$-graph $\Gamma$ with one unicursal component.
The homology group $H_{1}(\Gamma,\Z_{2})$ is generated by ``halves''
corresponding to vertices: for every vertex $v$ we have two halves
of the graph  $\Gamma_{v,1}$ and $\Gamma_{v,2}$, obtained by
smoothing at this vertex, see Fig. \ref{smo}. If the set of all
framed $4$-graphs (possibly, with some further decorations at
vertices) is endowed with a {\em parity}, we may assume that we are
given the following cohomology class $h$: for each of the halves
$\Gamma_{v,1},\Gamma_{v,2}$ we set
$h(\Gamma_{v,1})=h(\Gamma_{v,2})=p(v)$, where $p(v)$ is the parity
of the vertex $v$. Taking into account that every two halves sum up
to give the cycle generated by the whole graph, we have defined a
``characteristic'' cohomology class $h$ from $H^{1}(\Gamma,\Z_{2})$.

Let $c\in H_{1}(\Gamma,\Z_{2})$ be a subgraph of $\Gamma$. Denote by
$d(c)$ the collection of those vertices of $\Gamma$ where $c$ is
incident to exactly two half-edges, and these two edges are
non-opposite.

Then it  follows immediately that $c$ is equal to the sum $\sum_{i}
\Gamma_{d_{i},1}$ up to an overall addition of $\Gamma$.

\begin{figure}
\centering\includegraphics[width=180pt]{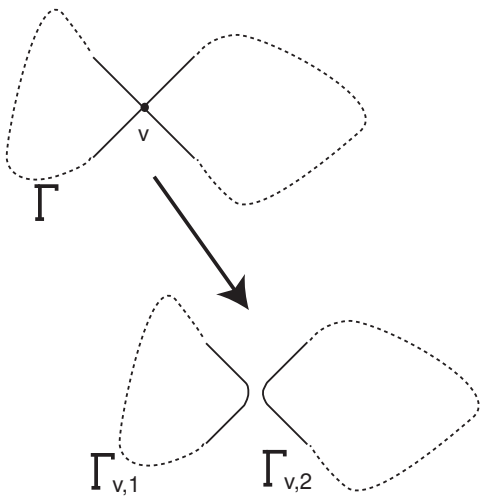} \caption{The halves
$\Gamma_{v,1}$ and $\Gamma_{v,2}$} \label{smo}
\end{figure}

 Collecting the properties of this cohomology
class and recalling the parity axiomatics, we see that

\begin{enumerate}

\item For every framed $4$-graph $\Gamma$ we have $h(\Gamma)=0$.

\item If $\Gamma'$ is obtained from  $\Gamma$ by a first
Reidemeister move adding a loop then for every basis
$\{\alpha_{i}\}$ of $H_{1}(\Gamma,\Z_{2})$ there exists a basis of
the group $H_{1}(\Gamma,\Z_{2})$ consisting of one element $\beta$
corresponding to the loop and a set of elements $\alpha'_{i}$
naturally corresponding to $\alpha_{i}$.

Then we have $h(\beta)=0$ and $h(\alpha_{i})=h(\alpha'_{i})$.

\item Let $\Gamma'$ be obtained from $\Gamma$ by a second increasing
Reidemeister move. Then for every basis $\{\alpha_{i}\}$ of
 $H_{1}(\Gamma,\Z_{2})$ there exists a basis in
$H_{1}(\Gamma',\Z_{2})$ consisting of one ``bigon'' $\gamma$, the
elements $\alpha'_{i}$ naturally corresponding to $\alpha_{i}$ and
one additional element $\delta$, see Fig. \ref{r2r3}, left.

\begin{figure}
\centering\includegraphics[width=220pt]{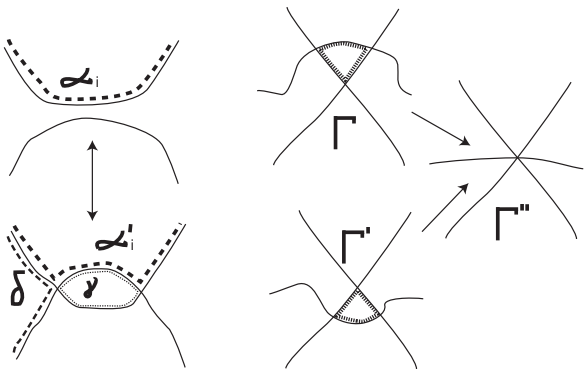} \caption{The
cohomology condition for Reidemeister moves} \label{r2r3}
\end{figure}

Then the following holds: $h(\alpha_{i})=h(\alpha'_{i})$,
$h(\gamma)=0$.

(note that we impose no constraints on $\delta$).

\item Let $\Gamma'$ be obtained from $\Gamma$ by a third
Reidemeister move. Then there exists a graph $\Gamma''$ with one
vertex of valency  $6$ and the other vertices of valency $4$ which
is obtained from either of $\Gamma$ or $\Gamma'$ by contracting the
``small'' triangle to the point. This generates the mappings
$i:H_{1}(\Gamma,\Z_{2})\to H_{1}(\Gamma'',\Z_{2})$ and
$i':H_{1}(\Gamma',\Z_{2})\to H_{1}(\Gamma'',\Z_{2})$, see Fig.
\ref{r2r3}, right.

Then the following holds: the cocycle $h$ is equal to zero for small
triangles, besides that if  for $a\in H_{1}(\Gamma,\Z_{2}),a'\in
H_{1}(\Gamma',\Z_{2})$ we have $i(a)=i'(a')$, then $h(a)=h(a')$.

\end{enumerate}

Thus, every parity for free knots generates some $\Z_{2}$-cohomology
class for all framed $4$-graphs with one unicursal component, and
this class behaves nicely under Reidemeister moves.

The converse is true as well. Assume we are given a certain
``universal'' $\Z_{2}$-cohomology class for all four-valent framed
graphs satisfying the conditions 1)-4) described above. Then it
originates from some {\em parity}. Indeed, it is sufficient to
define the parity of every vertex to be the parity of the
corresponding half. The choice of a particular half does not matter,
since the value of the cohomology class on the whole graph is zero.
One can easily check that parity axioms follow.

This point of view allows to find parities for those knots lying in
$\Z_{2}$-homologically nontrivial manifolds.

\section{Slice Genus and Cobordisms of Free Knots}

\begin{dfn}
Let $K$ be a framed $4$-valent graph. We say that $K$ has {\em slice
genus at most $h$} if there exists a surface ${\cal D}_{h}$ of genus
$h$ with one boundary component $S$, a $2$-complex $D_{h}\supset K$,
and a continuous map $g:{\cal D}_{h}\to D_{h}$, such that:

\begin{enumerate}

\item $g(\partial {\cal D}_{h})=K\subset D_{h}$; for every vertex
$v$ of $K$ we have $g^{-1}(v)=\{v_{1},v_{2}\}$, and small
neighbourhood $U(v_{1})\subset S$ is mapped to a pair of {\em
opposite edges} of $K$ at $v$;

\item the map $g$ is one-to-one everywhere except on the  union of
intervals: $\Sigma=\{x\in D_{h}:Card (g^{-1}(x))>1\}$;

\item the subset $\Sigma_{3}=\{x\in D_{h}: Card (g^{-1}(x))>2\}$ is
finite, and consists only of those points having exactly three
preimages; moreover $\Sigma_{3}\cap {\partial D}_{h}=\emptyset$;
double lines (preimages of $\Sigma\backslash \Sigma_{3}$ approach
$\partial{\cal D}_{h}$ transversally.

\end{enumerate}

The surface ${\cal D}_{h}$ will be called {\em the spanning surface}
or the {\em cobordism of genus $h$}.

In other words, we require that the free knot $K$ (image of the
circle $S$) is spanned by $D_{h}$, image of the $2$-surface ${\cal
D}_{h}$ with boundary $S$ and the singularities of the map $g:{\cal
D}_{h}\to D_{h}$ are all generic.

Analogously, one defines the cobordism of genus $h$ for
graph-links.

The relation $\sim$ for free-knots to be cobordant is defined for
cobordisms of genus $0$.

\end{dfn}

 The closure ${\bar \Sigma}$ contains also {\em cusps}:
those points $x\in D_{h}$ for which $Card(g^{-1}(x))=1$ and such
that for every small neighbourhood $U(x)$, the intersection
$U(x)\cap \Sigma$ is the punctured interval. Let $\zeta=g(\partial
{\cal D})$. Denote by $\Psi$ the preimage $g^{-1}(\bar
\Sigma)\subset {\cal D}_{h})$. Let $\Sigma_{2}$ be $\Sigma\backslash
\Sigma_{3}$.

The image $g(S)\subset D_{h}$ is a four-valent framed graph in
$D_{h}$. Indeed, this image is obtained from $S=\partial {\cal
D}_{h}$ by {\em gluing double points} of $S$; it has the framing
(opposite edge structure) induced from $S$: for a point $x$ in
$g(S)\cap \Sigma$, the preimage $g^{-1}(U(x)\cap g(S))$ consists of
two branches of $S$; the images of those two branches will generate
the two pairs of opposite edges.

\begin{dfn}
If $K$ admits a cobordism of genus $h$ and does not admit a
cobordism of genus $h-1$ we say that $K$ has {\em slice genus} $h$.
Notation: $sg(K)=h$.

A free knot of genus $0$ is called {\em null-cobordant} or {\em
slice}.\label{frknslice}
\end{dfn}

The following lemma follows from the definition of free knot:
\begin{lm}
If framed $4$-graphs $K,K'$ represent equivalent free knots then $K$
and $K'$ are cobordant and $sg(K)=sg(K')$.\label{sgcobord}
\end{lm}

Indeed, in Fig. \ref{rmovescobord} we demonstrate the cobordisms
between $K$ and $K'$: the first Reidemeister move corresponds to a
cusp point, the second Reidemeister move corresponds to a passage
through a tangency point, and the third Reidemeister move
corresponds to a triple point

Thus, it makes sense to speak about {\em the slice genus} of free
knots.

\begin{rk}
Let $K$ be a flat virtual knot and $|K|$ be the underlying free
knot. Then it follows from the definition that the slice genus of
$K$ is greater than or equal to the slice genus of $|K|$. In
particular, if $K$ is slice then so is $|K|$.
\end{rk}

\begin{ex}
The first example of a non-slice flat virtual knot was constructed
by Carter \cite{Ca}, it is shown in Fig. \ref{crt}.

\begin{figure}
\centering\includegraphics[width=100pt]{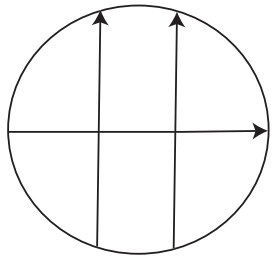} \caption{Carter's
non-slice flat virtual knot} \label{crt}
\end{figure}

In Fig. \ref{crt}, the arrows indicate the {\em clockwise direction
of branches}. Namely, orient the core circle of the chord diagram
counterclockwise and orient the immersed curve accordingly. If two
oriented branches $(a,b)$ of the curve  have an intersection at a
double point $X$ and the tangent vectors $v_{X,a}, v_{X,b}$ form a
positively-oriented basis, then the arrow is directed from $a$ to
$b$.

In this notation, for flat virtual knots, two chords pariticipating
in a second Reidemeister move should have opposite orientations. The
flat knot in Fig. \ref{crt} is non-trivial as a {\em flat virtual
knot}. Nevertheless, when we forget about the arrows and pass to the
free knot, we can first cancel the two vertical arrows (by a second
Reidemeister move) and then cancel the horizontal arrow (by a first
Reidemeister move). So, the corresponding  free knot $|K|$ is
trivial, and hence {\em slice}.
\end{ex}

So, if a free knot is {\em non-slice}, then so is {\em every} flat
virtual knot corresponding to it. The problem of finding {\em
non-slice} free knots is rather complicated. Another definition of
cobordisms (having another meaning), is given in \cite{IM} (it
follows Turaev's definition for cobordisms of words \cite{Tu}). In
this paper, two free knots are called {\em cobordant} if one can be
obtained from the other by a sequence of combinatorial moves. Some
invariants of such cobordisms are constructed, and non-sliceness of
knots is proved. We shall investigate the relation between those
{\em combinatorial} cobordisms and {\em topological} cobordisms
studied in the present paper, in a consequent paper.

In the work of Carter \cite{Ca} and Turaev \cite{Tu}, topological
sliceness obstructions for immersed curves were studied: for each
double point $x$ of an immersed curve $\Gamma$, one considers the
homology class of the halves $\Gamma_{x,1}$ for different vertices
$x$, and takes the homological pairing of these halves in the
surface. This approach can not be applied to free knots because a
framed four-valent graph is not assumed to be embedded in {\em any}
$2$-surface. Moreover, embeddings into different $2$-surfaces may
crucially change the intersection form for ``halves'' even with
$\Z_{2}$-coefficients.

\section{Parity of Curves in $2$-surfaces}

\label{lbl}

Let us now pay more attention to the structure of cobordisms of free
knots. Assume there is a cobordism $g:{\cal D}\to D$ (of genus zero)
spanning the free knot $K=g(\partial D)$.

Set $\Psi=g^{-1}({\bar \Sigma})$. Then $\Psi$ has a natural
stratification containing strata of dimensions zero and one. The
strata of dimension $0$ are {\em double points on the boundary},
{\em cusps}, and {\em triple points}. By a {\em double line} we mean
a minimal (with respect to inclusion) collection of $1$-dimensional
strata possessing the following properties:

1) two strata attaching the same cusp belong to the same double
line;

2) two strata attaching the same triple point from {\em opposite}
sides belong to the same double line, see Fig. \ref{dlinestpoints}.

\begin{figure}
\centering\includegraphics[width=200pt]{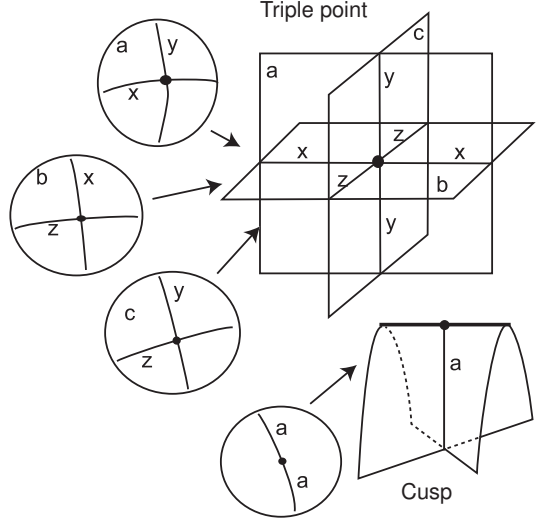}
\caption{Double lines $a,x,y,z$,cusps and triple points}
\label{dlinestpoints}
\end{figure}

Let $x\in K\cap \Sigma$ be a double point on the boundary of
$\partial D$. Assume $g^{-1}(x)=\{x_{1},x_{2}\}$.

Recall the definition of the Gaussian parity for a vertex $x$ of the
framed graph: we take an arc ${\iota}$ of $K$ connecting $x$ to $x$,
being the image of an arc ${\tilde \iota}$ on $S$,  and count the
parity of the double point number on ${\iota}$.

Now, let us consider the preimage ${\tilde \iota}\subset S$
connecting $x_{1}$ to $x_{2}$. Then $p(x)$ can be reformulated as
the parity of the set $Card(\tilde \iota\cap \Sigma)$. Note that
this line ${\tilde \iota}$ belongs to the boundary of the disc
${\cal D}$.

Note that in this definition we may choose ${\iota}$ to be any of
the two halves of the circle, and the resulting parity will not
depend on that choice. Moreover,  we take a path from $x_{1}$ to
$x_{2}$ locally directed towards the same half at both points, and
count the number of double points inside it.

Thus,  instead of ${\tilde \iota}\subset {\partial {\cal D}}$ we may
take an arbitrary path $\eta\subset  {\cal D}$ in generic position
to ${\bar \Sigma}$ connecting $x_1$ to $x_2$: we assume that
$\eta\cap {\bar \Sigma}=\eta\cap \Sigma_{2}$ consists only of
transverse intersections. All such curves lie on a disc ${\cal D}$
and have the same fixed ends, so they are all homotopic (rel.
boundary) to each other. The only thing we have to fix is the
behaviour of the curve in the neighbourhood of $x_{1}$ and $x_{2}$.
When we take ${\tilde \eta}\subset {\cal D}$, then the
neighbourhoods ${\tilde \eta}\cap U(x_{1})$ and ${\tilde \eta}\cap
U(x_{2})$ belong to the same half of the circle $S$. In terms of
${\cal D}$ and $D$, this can be reformulated as follows.

Let $\zeta\subset D$ be the $1$-stratum in $D$ attaching the point
$x$. Orient $\zeta$ arbitrarily, and orient the two preimages
$\zeta_{1}\cap U(x_{1})$ and $\zeta_{2}\cap U(x_{2})$ accordingly.
 Consider the two
vectors $v_{1}$ and $v_{2}$ tangent to ${\tilde \iota}$ at $x_{1}$
and $x_{2}$. We require that the bases $({\dot \zeta_{1}}, v_{1})$
and $({\dot \zeta_{2}}, v_{2})$ generate two different orientations
of the disc ${\cal D}$. If we change the direction of both $x_{1}$
and $x_{2}$ it will not change the parity of ${\tilde \iota}\cap
\Sigma_{2}$; here ${\dot \zeta_{i}}$ stays for the unit tangent
vector to $\zeta_{i}$, see Fig. \ref{paritaet}.

\begin{figure}
\centering\includegraphics[width=300pt]{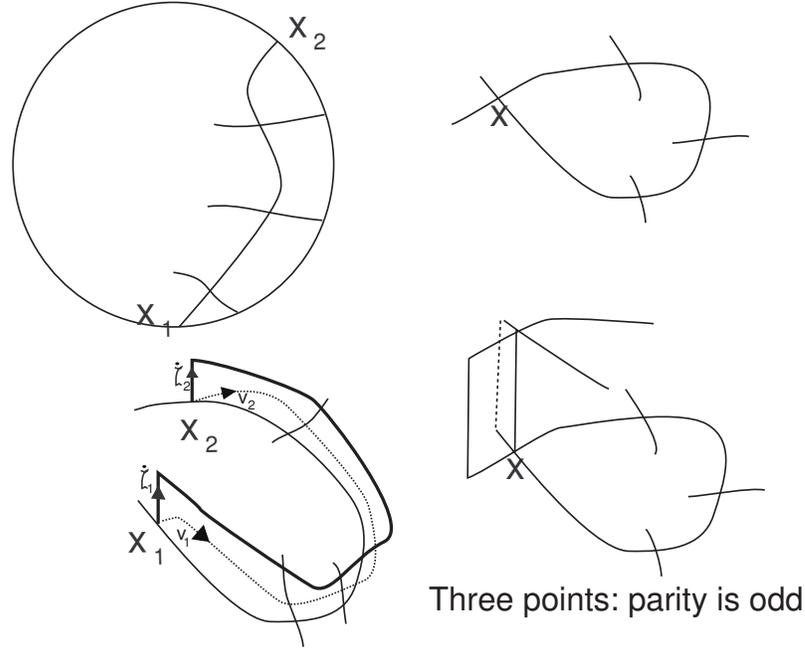} \caption{The
geometric way for defining parity} \label{paritaet}
\end{figure}

So, this gives another way to define a parity for $x$. Consider it
as the defintion of {\em any} double point.

Now, we can do the same for any arbitrary point on the double line
$x$. We call it {\em the Gaussian parity of a vertex} $x$.

The following Statement easily follows from the definition:
\begin{st}
The Gaussian parity is constant along double lines.\label{gp}
\end{st}

\begin{proof}
This is evident for two points belonging to the same $1$-stratum and
for points on two strata attaching the same cusp. When passing
through a triple point, the parity does not change, see Fig.
\ref{doesnotchange}. We see that the curve connecting the two
preimages of $A$ is parallel to the curve connecting the two
preimages of $B$ everywhere except for the two small domains; inside
these two domains, we have two intersections with double lines $p$
and $q$ which cancel each other.
\end{proof}

\begin{figure}
\centering\includegraphics[width=200pt]{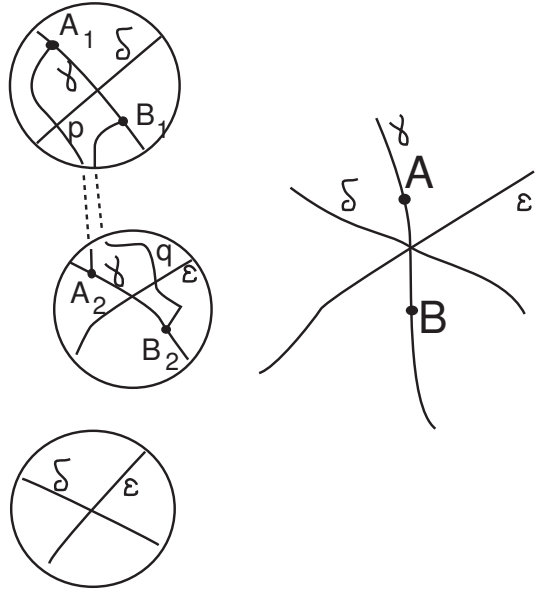}
\caption{The behaviour of parity and justified parity when passing
through a triple point} \label{doesnotchange}
\end{figure}

This leads to the definition of parity for a double line:
\begin{dfn}
Take an arbitrary generic (non-triple and not a cusp) point $x$ on a
double line $\gamma$, and consider the two preimages $x_1$ and $x_2$
of it on ${\cal D}$. Connect $x_1$ to $x_2$ by a generic path
${\tilde \eta}$, such that the behaviour of $\gamma$ in
neighbourhoods $U(x_1)$ and $U(x_2)$ is coordinated (as in the
definition of Gaussian parity). Now, the {\em Gaussian parity} of
the double line containing $x$ is the parity of the number of
intersection points between ${\tilde \eta}$ and $\Psi$.
\end{dfn}

This allows to define the set $\Psi_{even}$ to consist of all those
points of $\Psi$ beloning to {\em even} double lines.

Now, to define the {\em justified Gaussian parity}, we should take
an {\em odd double line} $\gamma$, an arbitrary generic point $x$ on
it, and consider the two preimages $x_{1}$ and $x_{2}$ of $x$. Then
we connect $x_1$ to $x_2$ by a generic path $\delta\in {\cal D}$ and
count the number of intersections between $\delta$ and
$\Psi_{even}$. If this number is even, we say that the $1$-stratum
of the odd double line containing $x$ is of the first type;
otherwise we say that this $1$-stratum is of the second type. Note
that the definition of {\em type} is even easier than the definition
of parity: we should not care about the local directions at the
endpoints since we do not count the number of intersections with
{\em odd} lines.

One immediately gets the following
\begin{st}
The Gaussian justified parity is constant on $1$-strata belonging to
$\Psi$. It changes from $b$ to $b'$ or from $b'$ to $b$ when passing
through a triple point formed by two odd double lines and one even
double line.\label{gpp}
\end{st}

\begin{proof}
The invariance along $1$-stratum is evident. Now view Fig.
\ref{doesnotchange}. Assume the double $\gamma$ and $\delta$ are odd
and the line $\epsilon$ is even. Then the two lines connecting the
preimages of $A$ and the preimages of $B$ are ``parallel'' except
for two domains where one of them passes through a double line at
$p$, and the other one passes through a double line at $q$. Since we
disregard intersection with {\em odd} double lines, we see that $q$
counts and $p$ does not.
\end{proof}

Having defined the {\em Gaussian parity} and {\em Gaussian justified
parity} in this way, we will be able to construct the invariant $L$
for any {\em section} of a cobordism $D$.

\section{Sliceness of Free Knots. The Main Theorem}

It turns out that the invariant $L$ of free knots is an {\em
obstruction to slicness}.

Before proving the main statement, we shall make several
observations concerning sliceness. By Lemma \ref{sgcobord}, it
follows that the slice genus is well defined on the set of free
knots.

The following statement is trivial.

\begin{st}
If a framed  $4$-graph $K'$ is equivalent to a framed $4$-graph $K$
(by Reidemeister moves), then the slice genus of $K'$ is equal to
that of $K$. In particular, if $K$ is slice then so is $K'$.
\end{st}

Thus, it makes sense to talk about {\em cobordisms of free knots},
not merely about cobordisms of framed $4$-graph.

\begin{figure}
\centering\includegraphics[width=200pt]{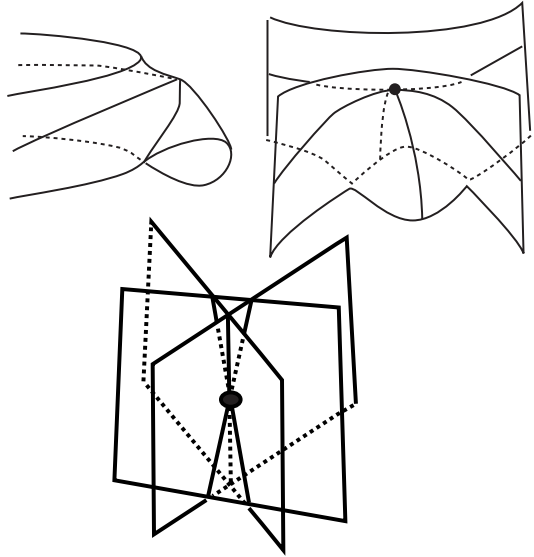}
\caption{Elementary cobordisms for Reidemeister moves}
\label{rmovescobord}
\end{figure}

As a corollary, we get the following
\begin{st}
If a framed $4$-graph $K$ is embeddable in $S^{2}$ or $T^{2}$ then
$K$ is slice.
\end{st}

Indeed, every framed $4$-graph on $S^{2}$ is equivalent (as a flat
knot) to the flat unknot; every framed $4$-graph on $T^{2}$ is
homotopic either to the trivial loop on the torus (with no
crossings) or to a non-separating loop without crossings (one may
think of such a loop as the longitude or the meridian in some
coordinate system), so, it is an unknot as well. Consequently, all
such knots are trivial in the free knot category.

\newcommand{\f}{{\cal F}}

\begin{st}
Let $K$ be a free knot, and let ${\f}(K)$ be a free knot obtained
from $K$ by deleting {\em odd} crossings. If $K$ is slice then so is
$\f(K)$.\label{deleteodd}
\end{st}

\begin{proof}
Indeed, any cobordism for $K$ generates a cobordism for $\f(K)$
obtained by separating all {\em odd double lines}: two points from
$\Sigma_{2}$ will be pasted together only if they lie on an even
double line. Note that this makes no contradiction with triple
points: a triple point either survives if it belongs to three even
double lines, or one sheet becomes disjoint from the two other ones
if two of the three double lines are odd.
\end{proof}

Denote the obtained cobordism for $\f(K)$ by $\f(D)$.

The main result of the present paper is the following

\begin{thm}[Main Theorem]
If a free knot $K$ has $L(K)\neq 0$, then $K$ is not
slice.\label{mthm}
\end{thm}

In particular, this yields the following

\begin{crl}
Let  $\gamma$ be a curve immersed in an oriented closed $2$-surface
$S_{g}$. Then if for a free knot $\Gamma$ corresponding to $\gamma$
one has $L(\Gamma)\neq 0$ then the underlying $\gamma$ is not slice
as a flat virtual knot.
\end{crl}

Indeed, a general position image of the disc in a $3$-manifold  is a
spanning disc having only three-dimensional singularities. The
converse statement is, generally, not true.

The problem of finding obstructs for a surface $S_{g}$ with a curve
$\gamma$ to span a disc immersed in a $3$-manifold was studied by
Carter \cite{Ca}, Turaev,  \cite{Tu} etc. Some topological
obstructions based on homology of $S_{g}$ were constructed.

In the present paper we consider a more complicated problem: instead
of curves in $2$-surfaces we consider framed four-valent graphs, and
instead of spanning $2$-discs in $3$-manifolds we consider
``abstract'' spanning $2$-discs. However, as we see from the above
discussion concerning parity and homology and the formulation of the
main theorem, the notion of parity which plays the role of
``subsitute of homology of $S_g$''.

\subsection{Constructing the Morse function and the Reeb graph}

The proof will consist of several steps.

We shall adopt the following notation for the maps: by $g$ we shall
mean the map ${\cal D}\to D$ corresponding to the cobordism, and by
$f$ we shall denote either of the two maps: the Morse function
$f:D\to \R$ (see definition ahead) and the composition $g\circ
f:{\cal D}\to \R$ will be also denoted by $f$ (abusing notation).

Assume the knot $K$ is slice and admits a cobordism  $g:{\cal D}\to
D$ (of genus zero).

\begin{dfn}
By a {\em Morse function} on $D$ we mean a Morse function $f:{\cal
D}\to [0,\infty)$ such that if $g(x)=g(y)$ then $f(x)=f(y)$, all
Reidemeister move points:  triple points, cusp points and tangency
points on ${\cal D}$, lie on non-critical levels of $f$, and
$f^{-1}(0)=K, f^{-1}(1)=\emptyset$. By abuse of notation we shall
denote the function on ${\cal D}$ and the function on $D$ by the
same letter $f$. By a {\em non-singular} value of the function $f$
we mean a noncritical value $X$ such that $f^{-1}(X)\subset {\cal
D}$ contains no cusps and no triple points. A Morse function on $D$
will be called {\em simple} if every singular level contains either
exactly one critical point or exactly one triple point or exactly
one cusp point.
\end{dfn}

From now on, we require that the Morse function of $D$ is simple and
the level $0$ is non-singular. It is clear that such Morse functions
are {\em everywhere dense in the class of all functions}. Every
Morse function has singular levels of two types: those corresponding
to Morse bifurcations (saddles, minima, and maxima) and those
corresponding to Reidemeister moves (corresponding to cusps,
tangency points and triple points). Denote singular levels of the
function $f$ by $c_{1}<\dots<c_{k}$ and choose non-critical levels
$a_{i}$: $0=a_{0}<c_{1}<a_{1}<c_{2}<\dots a_{k}<c_{k}<a_{k+1}=1$.

Let us construct the {\em Reeb graph} $\Gamma_{f}$ (molecula) of the
function $f$ as follows.  The univalent vertices of the Reeb graph
will correspond to minima and maxima of the function $f$; the
vertices of valency three will correspond to saddle points; edges
will connect critical points; every edge will correspond to a
cylinder $S^{1}\times I\subset {\cal D}$ which is continuously
mapped by $f$ to a closed interval between some two critical point;
this cylinder has no Morse critical points inside. One edge will
emanate from the point $0$ corresponding to the circle $S=\partial
{\cal D}$, see Fig. \ref{Reeb}.

\begin{figure}
\centering\includegraphics[width=200pt]{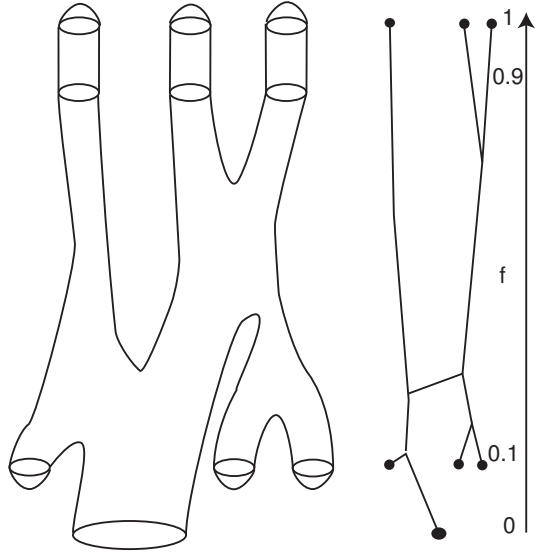} \caption{The Reeb
graph and circle bifurcations} \label{Reeb}
\end{figure}

Since this graph is a Reeb graph of the Morse function on a disc
${\cal D}$, the graph $\Gamma_{f}$ is a {\em tree}.

Our next goal is to endow each edge of the Reeb graph with a
non-negative integer {\em label}. The label of the edge emanating
from $0$ will coincide with $L(K)$.

For every non-singular level $c$ of $f$ the preimage
$K_{c}=f^{-1}(c)\subset D$ is a free link; when passing through a
Reidemeister singular point, the link $K_c$ is operated on by the
corresponding Reidemeister move; when passing through a Morse
critical point it gets operated on by a Morse-type bifurcation.
Every crossing of $K_{c}$ belongs to some double line of $D$. Define
the {\em parity}  of the crossing to be the Gaussian parity of the
double line, it belongs to. Analogously, define the {\em justified
parity} of a crossing to be that of the stratum it belongs to. One
can easily see that if a section of the Morse function is a framed
$4$-valent graph with one unicursal component then the parity and
justified parity coincide with the Gaussian parity and justified
parity defined directly via Gauss diagram.

Consider the free link $K_{c}$ and orient its components arbitrarily
(as we shall see further, the orientation will be immaterial); for
every unicursal component $K_{c,j}$ of the free link $K_{c}$ we may
define the conjugacy class $\delta(K_{c,j})$ of $\gamma(K_{c,j})$ in
$G$ just as it is done for free knots with Gaussian parity and
justified parity. Let $\delta(K_{c})$ be the unordered collection of
all $\delta(K_{c,j})$ for all $j$ (with repetitions).

From Statements \ref{gp} and \ref{gpp} we get the following
\begin{lm}
The parity and justified parity defined on the set of $K_{c}$
satisfy the parity and justified parity axioms under those
Reidemeister moves which happen within the cobordism $D$.
\end{lm}

This lemma, in turn, yields the following
\begin{lm}
For an interval $[a,b]\subset [0,1]$ containing Reidemeister
singular points and no Morse critical points (for a coordinated
choice of orientations of link components) we have
$\delta(K_{a})=\delta(K_{b})$ as collections of conjugacy classes of
elements from $G$ counted with multiplicities.

Moreover, $\delta(K_{0})$ is the conjugacy class of $\gamma(K)$.
\end{lm}

The proof literally repeats the proof of Theorem \ref{th1} applied
to each unicursal component of the link.

Now, we would like to treat $\delta$ as a collection of {\em
non-negative integers} rather than just conjugacy classes and to
forget about orientations of components of $K_{c}$. To this end, we
prove the following

\begin{lm}
Let $c$ be a non-singular level of $f$, and let  $K_{c,1}, \dots,
K_{c,n}$ be  unicursal components of the free link
$K_{c}=f^{-1}(c)\subset D$. Then for every $i=1,\dots,n$:

1) The total number of intersection points between $K_{c,i}$ and
$K_{c,j}, j\neq i$, is even.

2) The number of {\em odd} intersection points between $K_{c,i}$ and
$K_{c,j}, j\neq i$, is even.\label{evenodd}
\end{lm}

\begin{proof}
The proof follows from the fact that the preimage of $K_{c,i}$ in
${\cal D}$ is a circle, and the intersection of a closed circle with
a set $\Psi$ (or $\Psi_{even}$) in ${\cal D}$ consists of an even
number of points.
\end{proof}

Lemma \ref{evenodd} immediately yields that every $\gamma(K_{c,i})$
for a non-singular $c$  is represented by an element $(0,l)$  on the
Cayley graph of $G$.

Let $L_{c}=\{l_{c,1},\dots, l_{c,m}\}$ be the collection of
integers (with repetitions) obtained from $\delta(K,c)$ by replacing
conjugacy classes with absolute values of their second coordinates.

Each $l_{c,i}\in \N\cup \{0\}$ corresponds to a component of the
free link $K_{c}$ and does not change under Reidemeister move when
changing $c$ without passing through Morse critical points.
Associate it with the corresponding edge of the graph $\Gamma_{f}$.

Now, let us analyze the behaviour of these labels $l_{c,i}$ at
vertices of the graph $\Gamma_{f}$.

\begin{lm}
Assume $K_{c-\eps}$ and $K_{c+\eps}$ differ by one Morse bifurcation
at the level $c$.

Then:

\begin{enumerate}

\item If this bifurcation corresponds to a birth of a circle then
$L_{c+\eps}$ is obtained from  $L_{c-\eps}$ by an addition of $0$;

\item If it corresponds to a death of a circle then $L_{c+\eps}$ is
obtained from $L_{c-\eps}$ by a removal of $0$;

\item In the case of fusion all elements of $L_{c-\eps}$ except two ones ($m$ and
$n$) remain the same, and the elements $m$ and $n$ turn into some
$k=\pm n\pm m$ to form an element of $L_{c+\eps}$.

\item The fission operation is the inverse to the fusion: instead of
one element $k$ one gets a pair of elements $m,n$ such that $\pm
m\pm n=k$.\label{klm}
\end{enumerate}

\end{lm}

\begin{proof}
The first two assertions are obvious: a trivial circle has no
points, thus the corresponding element of $G$ is the unit of $G$ and
the label is equal to $0$.

The last two assertions follow from the following observation. If a
circle $C$ with a marked point $X$ splits into two circles by a
bifurcation connecting $X$ to some point $Y$, then the corresponding
word $w\in G$ splits into the product $w=w_{XY}w_{YX}$.

The rest of the proof follows from the multiplication rule in $G$:
for elements $u,v\in G$ having coordinates $(u_{1},u_{2})$ and
$(v_{1},v_{2})$, respectively, the product $u\cdot v$ has
coordinates $(\pm u_{1}\pm v_{1},\pm u_{2}\pm v_{2})$.

\end{proof}

This leads to the following way of proving the main theorem. The
graph $\Gamma_{f}$ has all vertices except one (corresponding to the
initial knot $K$) having label $0$. At each vertex, the two labels
(with signs $\pm$) sum up to give the third label. Thus, taking into
account that $\Gamma_{f}$ is a tree, we see that the label of the
initial vertex is $0$, so $L(K)=0$. A contrapositive completes the
proof of Theorem \ref{mthm}.

\begin{ex}
Consider the free knot $K_1$ shown in Fig. \ref{fig1}. By Theorem
\ref{mthm} it is not-slice. Thus, all flat knots $K$ with underlying
free knot $K_1$ are not slice, either.
\end{ex}

\section{Cobordisms of Higher Genus}

The methods used for proving that $L$ is a sliceness obstruction, do
not work immediately for estimating the slice genus. There are two
reasons when we use the fact that the spanning surface is the
sphere. First, when we define the {\em parity} and {\em justified
parity} of the double lines on ${\cal D}$, we use an arbitrary curve
connecting some two points (with some constraints on the direction
at the final points) and say that all such curves are homotopic. For
arbitrary spanning surface ${\cal D}_{h}\to D_{h}$ it is not the
case.

So, in order to define the parity of double lines, one will have to
impose some obstructions on the spanning surface: one has to require
that the cohomology class dual to $\Psi$ is $\Z_{2}$-trivial. The
matter of $\Z_{2}$-trivial homology classes and even-valent graphs
in knot theory is closely related to {\em atoms}, see \cite{Atoms}.

The other problem comes from the fact that the Reeb graph
corresponding to a Morse function lying on a surface having non-zero
genus, is not necessarily a tree.

So, starting with a free knot $K$ with, say, $L(K)=8$, one may
(principally) split this free knot into the free link consisting of
two free knots $K_{1}$ and $K_{2}$ with $L(K_{i})=4$ and then merge
these free knots to get an unknot (since $4$ and $-4$ sum up to give
$0$ which is the value of $L$ on the unknot).

In the present section, we say that how to overcome these two
difficulties for cobordisms (of arbitrary genus) of certain type.

Let ${\cal D}_{g}$ be a surface with boundary $S^1$, and let
$\Sigma_{2}$ be the set of double points of ${\cal D}_{g}$.
Obviously, $\Sigma_{2}$ defines a relative $\Z_{2}$-homology class
$\kappa$ of ${\cal D}_{g}\slash S^{1}$. This homology class is an
obstruction for the surface to be checkerboard-colourable; also,
this is an obstruction for well-definiteness of even/odd double
lines.

Namely, if we look at the definition of an even/odd double line, we
see that there is an ambiguity in the choice of path connecting two
preimages of a generic point on the double line. For the case of a
disc cobordism, the parity of double lines is well defined, because
all such curves are homotopic. For ${\cal D}_{g}$ the unique
obstruction to this well-definiteness is the class $\gamma$.

We call a cobordism of genus $g$ {\em checkerboard} (or {\em
atomic}) if the corresponding class $\gamma$ vanishes.

The next task (after detecting which double line is even and which
one is odd) is to make a distinction between $b$ and $b'$. To this
end, one should do the same for preimages of points lying on odd
double lines, connect them by a generic curve, and count the
intersection with {\em even} double lines. So, we see that the only
obstruction is the relative $\Z_{2}$-homology class $\gamma'$ of
${\cal D}_{g}\slash S^{1}$ generated by {\em even double lines}.

An example of an atomic (checkerboard) but not $2$-atomic cobordism
is shown in Fig. \ref{at2at}. Here the torus with a disc removed is
presented by a square (opposite edges are assumed identified). The
three pairs of lines (a dashed pair of lines, a thin line pair and a
thick line pair) represent the collection of double lines of the
cobordism. There is one triple point formed by the mutual
intersection of these lines. One can see that the torus with the
disc removed is checkerboard colourable, and that solid lines are
both odd, the dashed line is even. However, if we remove odd lines,
we are left with dashed lines, and the picture is not checkerboard
colorable any more.

\begin{figure}
\centering\includegraphics[width=200pt]{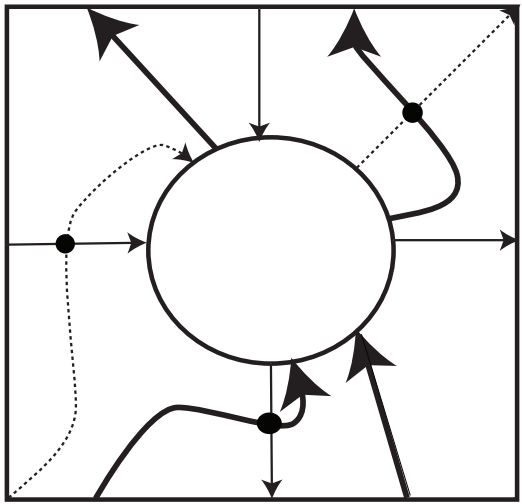} \caption{An
atomic cobordism which is not $2$-atomic} \label{at2at}
\end{figure}

We say that a checkerboard cobordism is {\em $2$-atomic} if
$\gamma'$ vanishes.

The aim of the present section is to prove the following

\begin{thm}
Assume for a $1$-component framed $4$-graph $K$ we have $L(K)\neq
0$. Then there is no {\em $2$-atomic} cobordism spanning the knot
$K$ of any genus.\label{newth}
\end{thm}

\begin{proof}
Let us revisit the proof of Theorem \ref{mthm}. Assume there is a
$2$-atomic cobordism ${\cal D}_{g}\to {\cal D}$ of some genus $g$
spanning the knot $K$. Fix a generic Morse function $f$ on ${\cal
D}_{g}$ (the corresponding Morse function on  $D_{g}$ will be
denoted by the same letter $f$).

Following the lines of the proof of Theorem \ref{mthm}, we see that:

1) at each non-critical level of $f$ we have a free link $L_{c}$ of
some number of components, and with each component we associate a
natural number coming from a conjugacy class in $G$.

2) these numbers behave nicely under Morse bifurcations, i.e., a
birth/death of a circle corresponds to an addition/removal of an
occurency of $0$;

3) for a saddle point the three numbers $k,l,m$ corresponding to the
adjacent edges satisfy $\pm k\pm l\pm m=0$.

Note that every saddle point merges two circles into one or splits
one circle into two circles (the Moebius bifurcation is impossible
because ${\cal D}_{g}$ is orientable).

Recall that each of these numbers $k$, $l$, $m$ at this point is
defined only up to sign (equivalently, we have an absolute value of
these three numbers), and this is not sufficient to prove the
theorem in the case of cobordism of arbitrary genus.

Now, let us be more specific and study these numbers on edges in
more detail. Let $t$ be a non-critical level of $f$, and let
$K_{1},\dots, K_{n}$ be the free knots composing the corresponding
free link $K_{t}$. For each of these knots $K_{j}$ we have defined
the integer number by taking an initial point of $K_{j}$. This
number $2l$ comes from an element of $G$ of the form $x=(bb')^{l}$,
and we see that a conjugation of $x$ by any of $a,b,b'$ takes $x$ to
$(b'b)^{l}=x^{-1}$. So, a conjugation by a word of {\em even length}
does not change the number $2l$, whence the conjugation by a word of
an {\em odd length} takes it to $-2l$.

Let us take a checkerboard colouring of ${\cal D}_{g}$ with respect
to the cell decomposition generated by double lines.

Whenever we take an initial point of any section, this initial point
does not belong to any double line, so, it has some colour, black or
white. We see that a conjugation by a word of even length does not
change the colour of the initial point.

This means for every component $K_{j}$, the numbers $l_{b}(K_{j})$
and $l_{w}(K_{j})$ are well defined and
$l_{b}(K_{j})=-l_{w}(K_{j})$, where subsripts $b$ and $w$ correspond
to the colour choice of the initial point on the circle.

Fix the colour black once forever. Then every edge of the Reeb graph
acquires an integer number $l_{b}$, and for every level when we
merge/split circles, the sum of these $l_{b}$'s does not change.

So, the sum $l_b$ remains invariant for every non-critical level of
the Morse function. Since it is non-zero at $t=0$, it will remain
non-zero for every $t$.

This completes the proof of theorem \ref{newth}.

\end{proof}

\section{Further Directions and Unsolved Problems}

\subsection{2-knots in 4-space and 4-manifolds}

The parity we have defined for double lines for the case of a disc
(see section \ref{lbl}) (or a surface with boundary) can be defined
just in the same way for a generic map from a $2$-surface to a
$4$-manifold.

According to \cite{Ma,Af}, lots of invariants of virtual knots can
be refined by using {\em parity}. Probably, invariants of $2$-knots
in $4$-spaces (say, the quandle) can be refined in the same way by
using parity of double lines for $2$-surfaces.

\subsection{Gaussian parity and other parities}
Another question is whether the same invariant $L$ counted by using
another parity and justified parity (not Gaussian) can lead to the
same result. This question is closely related to the following
question:

Are there any other parities for double lines of $2$-surfaces with
singularities and which of them can be obtained as extensions from
usual parities of framed $4$-graphs?

\subsection{Virtual knots and stronger groups}

In \cite{MM}, a generalization of the group $G$ is constructed; one
uses {\em iterated parities} by applying the map $\f$ and counting
parities of the obtained chord diagram. In a way similar to the
described above, one obtains an invariant of free knots valued in a
certain more complicated group. Seemingly, this group will give
further obstructions for sliceness; moreover, for {\em flat knots}
and {\em virtual knots}, the group information can be enriched by
adding some information about crossings (signs, orientations, etc).

Besides, when we deal not merely with free knots, but with some more
complicated objects (virtual knots), we may use more complicated
groups \cite{Fintype}. Most probably, the non-triviality of elements
of these groups may yield sliceness obstruction for virtual knots.

This problem is especially actual because the sliceness in dimension
$4$ has different meanings: a piecewise-flat slice knot might not be
smoothly slice. So, it would be very interesting to get some {\em
smooth sliceness obstructions}


\begin{thebibliography}{100}


\bibitem{Af} D.M.Afanasiev, {\em On strengthening tnvariants of
virtual knots by using parity}, {\em Mat. Sb.}, to appear.

\bibitem{Ca} J.S.Carter (1991), {\em Closed Curves that never extend to
proper maps of disks}, {\em Proc. Amer. Math. Soc.}, {\bf 113}, No.
3, pp. 879-888.

\bibitem{Gib} A.Gibson (2009) {\em Homotopy Invariants of Gauss
words}, Arxiv:Math.GT./0902.0062

\bibitem{IM} D.P.Ilyutko, V.O.Manturov, {\em Cobordisms of Free
Knots and Gauss words}, Arxiv:Math.GT/09042862



\bibitem{Ma} V.O.Manturov (2010), {\em Parity in Knot Theory},
{\bf Mat. Sb.}, 2010, {\bf 201}, 5,  pp. 65–110

\bibitem{Ma1} V.O.Manturov, {\em On Free Knots},
ArXiv:Math.GT/0901.2214 v2

\bibitem{Atoms} V.O.Manturov, {\em On Free Knots and Links}, ArXiv:
Math.GT/0902.0127

\bibitem{Paritytrieste} V.O.Manturov, {\em Free Knots and Parity}, Arxiv:Math.GT/09125348,
v.1., to appear in: {\em Proceedings of the Advanced Summer School
on Knot Theory, Trieste}, Series of Knots and Everything, World
Scientific.

\bibitem{Fintype} V.O.Manturov, {\em Free Knots, Groups, and
Finite-Type Invariants}, ArXiv:Math.GT/1004.4325

\bibitem{MM} V.O.Manturov, O.V.Manturov, {\em Free Knots and Groups},
{\em J. Knot Theory Ramif.}, to appear.

\bibitem{Flats} Turaev, V.G., {\em Virtual open strings and their cobordisms}, preprint, 2004, arXiv:
math.GT/ 0311185 v5

 \bibitem{Tu} Turaev, V.G., {\em Cobordisms of Words},
Arxiv:Math.CO/0511513, v.2.

\bibitem{Tu2} Turaev, V.G., {\em Cobordisms of Knots on Surfaces},
Arxiv:Math.GT/0703055, v.1.



\end{thebibliography}
\end{document}